\pdfoutput=1
\documentclass[a4paper]{amsart}

\usepackage{amsmath}
\usepackage{amsfonts}
\usepackage{amssymb}
\usepackage[all]{xypic}
\usepackage{hyperref}
\usepackage{amsmath,amsfonts,amsthm,epsfig,amsthm}
\usepackage{graphicx,epsfig,color}

\input cyracc.def
\font \tencyr=wncyr10
\newfam\cyrfam
\font\tencyr=wncyr10
\font\sevencyr=wncyr7
\font\fivecyr=wncyr5

\textfont\cyrfam=\tencyr \scriptfont\cyrfam=\sevencyr
\scriptscriptfont\cyrfam=\fivecyr

\let\kappa\varkappa

\newtheorem{proposition}{Proposition}
\newtheorem{definition}{Definition}
\newtheorem{lemma}{Lemma}
\newtheorem{theorem}{Theorem}
\newtheorem{corollary}{Corollary}
\newtheorem{remark}{Remark}
\newtheorem{example}{Example}
\newtheorem{fact}{Fact}

\newcommand{\N}{\mathbb{N}}
\newcommand{\R}{\mathbb{R}}

\newcommand{\C}{\mathcal{C}}\newcommand{\D}{\mathcal{D}}
\newcommand{\A}{\mathcal{A}}
\newcommand{\E}{\mathcal{E}}

\renewcommand{\O}{\mathsf{O}}

\newcommand{\I}{\mathcal{I}}
\newcommand{\II}{\check{\I}}
\newcommand{\J}{\check{J}}
\newcommand{\III}{\I^{(\infty)}}

\newcommand{\Th}{^\textrm{th}}
\newcommand{\St}{^\textrm{st}}

\renewcommand{\S}{\boldsymbol{S}}

\newcommand{\ddd}{\boldsymbol{d}}\newcommand{\p}{\boldsymbol{p}}\newcommand{\n}{\boldsymbol{n}}

\newcommand{\DD}{\boldsymbol{D}} \newcommand{\CC}{\boldsymbol{C}}

\newcommand{\Hom}{\mathrm{Hom\,}}

\newcommand{\REF}[1]{{\normalfont (\ref{#1})}}

\newcommand{\Span}[1]{\langle\, #1\, \rangle}

\newcommand{\GL}{\mathrm{GL\,}}

\newcommand{\Gr}{\mathrm{Gr\,}}

\newcommand{\df}{\stackrel{\mathrm{def}}{=}}
\newcommand{\loc}{\stackrel{\mathrm{loc}}{=}}

\newcommand{\virg}[1]{``#1''}
\renewcommand{\P}{\mathbb{P}}

\newcommand{\rank}{\mathrm{rank\,}}

 \title{The geometry of the  space of Cauchy data of nonlinear PDEs}

\author{G.~Moreno}  

       \address{Mathematical Institute in Opava\\
Silesian University in Opava\\
Na Rybnicku 626/1, 746 01 Opava, Czech Republic.}
  \email[G.~Moreno]{Giovanni.Moreno@math.slu.cz}

\begin{document}
          
          \begin{abstract}
          
 First--order jet bundles can be put at the foundations of the modern geometric approach to nonlinear PDEs, since higher--order jet bundles can be seen as constrained iterated jet bundles. The definition of first--order jet bundles can be given in many equivalent ways---for instance, by means of Grassmann bundles. In this paper we generalize it by   means of flag bundles, and develop the corresponding theory for higher--oder and infinite--order jet bundles. We show that this is a natural geometric framework for the space of Cauchy data for nonlinear PDEs. As an example, we derive a general notion of transversality conditions in the Calculus of Variations.

         \end{abstract}


\maketitle

\setcounter{tocdepth}{1}

\tableofcontents

%
%
%

\section{Introduction}
Algebraic Geometry entered  adulthood when its intellectual energies, traditionally committed to find  concrete solutions of algebraic equations (i.e., \emph{points} of an arithmetic space), began to  wonder about the structure of the equations themselves (i.e.,  \emph{ideals} in rings over arithmetic fields).  The   theory  of nonlinear PDEs   underwent a similar development,  though highly ramified and  dependent on the  intermittent and diversified impulses coming from natural sciences, and   it is still inappropriate to speak about \virg{the} theory of nonlinear PDEs, for none of the proposed frameworks was  enthusiastically embraced by the mainstream. The reader may find relevant  historical information, as well as an exhaustive list of references in the  2010 review  \cite{KrasVer}. \par
This paper is committed to the perspective that (smooth) solutions of a (regular enough) system of nonlinear PDEs in $n$ independent variables (henceforth called \emph{equation}, for short) are to be interpreted as the maximal integral submanifolds (henceforth called \emph{leaves}) of an $n$--dimensional   involutive distribution on a pro--finite manifold, and adheres to the philosophy that relevant invariants of the equation are encoded by an appropriate    cohomological theory, possibly twisted with nonlocal coefficients, called \emph{characteristic} or \emph{leafwise} cohomology (of the equation). In such a framework, the space of leaves itself, which is (as a rule) quite bad--behaved, can be put aside, and  the focus   diverted to the characteristic cohomologies of the foliation. We shall use the word \virg{secondary} (following \cite{Vin2001}) as a synonimous of \virg{leafwise} in the pro--finite context\footnote{We keep the distinction between foliations of finite and pro--finite manifolds, due to the failure of the key Frobenious theorem on the latter.} and we adopt the same framework and terminology which can be found, e.g.,  in the introductory section of \cite{Luca} (for example, a \emph{secondary point} is just  a leaf, a \emph{secondary manifold} is the leaf space of a foliation over a pro--finite manifold, a \emph{secondary map} is a map preserving leaves, etc). 
%
%
\begin{eqnarray}
\epsfig{file=SottovarietaDimUno.pdf,width=40mm}&&\epsfig{file=GettiDimUno.pdf,width=40mm}\nonumber\\
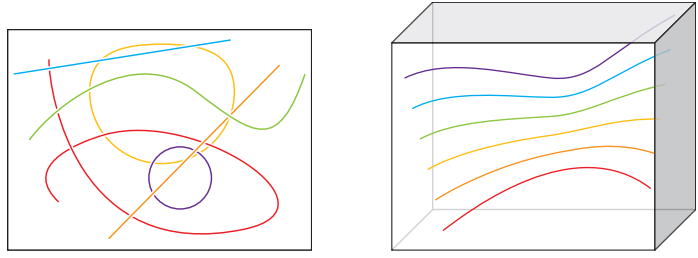
\textrm{$n$--dimensional submanifolds of $E$}&  {\Leftrightarrow}  & \textrm{Leaves of  $J^\infty(E,n)$}\label{eqPuntiSecondati}
\end{eqnarray}
So, the \virg{solution space} of an equation can be seen as a secondary submanifold of  the \emph{empty equation}\footnote{This is the reason why, sometimes, a leaf of $J^\infty(E,n)$ is also called a \virg{solution}.} $J^\infty(E,n)$, since the graphs of the solutions of the former correspond to the leaves  of the latter \REF{eqPuntiSecondati}. Nonetheless, it is quite evident  that the same equation dictates restrictions also on non--maximal integral submanifolds: indeed, by definition, a non--maximal integral submanifold is contained into a leaf, and among leaves there are the solutions.  In this paper we propose a very natural geometric framework where $(n-1)$--dimensional integral submanifolds (henceforth called \emph{small leaves}) coexist with the maximal ones, study its structure, and reveal some interesting properties of its characteristic cohomology.    Small leaves  are nothing but the geometric counterparts of infinite--order Cauchy data (Section \ref{secCauchDat}), taking prominent roles in the theory of nonlinear PDEs, calculus of variations, field theory, etc., and, in our approach, they can be treated on the same footing as solutions. To this end, it is compulsory  to \virg{nest} one jet space into another \REF{eqGettiInnestati}, much as, in another context,  flag manifolds are constructed out of nested Grassmmannians.
%
 \begin{eqnarray}
\epsfig{file=SottovarietaSuFogliaDimUno.pdf,width=40mm}&&\epsfig{file=GettiSuFogliaDimUno.pdf,width=40mm}\nonumber\\
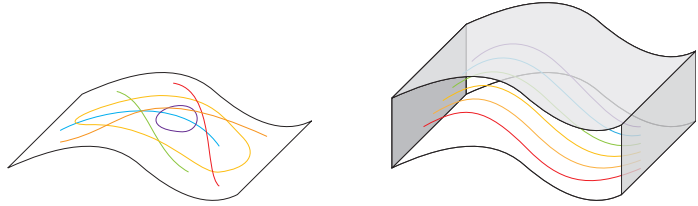
\textrm{$(n-1)$--dim. sub. of a leaf  $L\subset J^\infty(E,n)$}&  {\Leftrightarrow}  & \textrm{Leaves of  $J^\infty(L,n-1)$}\label{eqGettiInnestati}
\end{eqnarray}
Once Cauchy data and solutions of a PDE  are framed in the same secondary context, 
 it becomes natural to perform algebro--geometric manipulations which mix   secondary notions of horizontal degree $n$ with   ones of horizontal degree $n-1$. For example, a boundary variational   integral (i.e., a secondary function of horizontal degree $n-1$)   can be combined with a variational integral (i.e., a secondary function of horizontal degree $n$), and from their interaction it arises, in a surprisingly straightforward way,   a  general notion of transversality conditions (Section \ref{SecEsempioVariazionale}). \par
 %
\subsubsection*{Structure of the paper}
In Section \ref{secFunctJey} we define special subsets of the jet bundle, needed to associate with a map between manifolds a map between the corresponding jet bundles. This will allow to speak of \virg{projectable} and, in particular, of \virg{horizontal} jets later on, and hence to be able to deal with  the jet bundles over pro--finite manifolds.\par
Section \ref{secFlagGr} contains the well--known material about Grassmannian and flag manifolds, with the focus on the universal sequence associated to a Grassmannian and the canonical bundles over  flag manifolds.  These   notions are at the heart of the definition of $1\St$ order jet bundles and flag jet bundles, respectively.\par
In Section \ref{secInvPlanes} we introduce a class of equations (given, in coordinates, by \REF{eqFunzLocDistrib1} and \REF{eqFunzLocDistrib2}) which, for $n$ independent variables, constitute the key ingredient to define higher--order jet bundles out of lower--order ones, and, for $n-1$ independent variables, lead straightforwardly to the notion of a jet of a Cauchy datum. These are but  examples of   equations of involutive planes of a distribution.\par
Inheritance of involutivity allows to mimic the definition of a flag manifold and to introduce, in a similar fashion, higher--order flag jet bundles $J^k(E,n,n-1)$. In Section \ref{secFlagJet}, besides the conceptual definition,   two natural coordinate systems are proposed, stemming from the fundamental equation  \REF{eqEquazioniCoordinateDatiFinitiEQ}, which will be useful later on for   the description of the    canonical bundles associated with  $J^k(E,n,n-1)$.   \par
The notion of the $1\St$ jet of an $(n-1)$--dimensional involutive plane is \virg{almost} the same as the notion of a flag jet, were, as usual in the theory of jet bundles, \virg{almost} means that the desired property holds correctly only on the inverse limit, i.e., for infinite jets. Section   \ref{secFlagJetInvpla} clarifies this relationship through the fundamental diagram \REF{eqDiagrammaLibroAperto}.  The advantage with this new point of view is that involutive planes, unlike flag jets, are naturally understood as solutions of an equation,  which we denote by $\II_{n-1}(\C)$.\par
Having introduced flag jets was not a mere exercise, since they carry a natural normal bundle, which is essential to discover   the structure of the space of Cauchy data. The idea, sketched  in Section  \ref{remNormBundleUno} by introducing the bundle of infinite--order normal directions, is that the space of Cauchy data can be seen as the space of sections of an (infinite--dimensional) bundle over a fixed Cauchy value, whose fiber coordinates capture the ideas of \virg{purely normal derivatives}.\par
In Section \ref{secCauchDat}, after having given  the formal definitions of finite and infinite--order Cauchy data, it is shown how a higher--order Cauchy datum can be constructed \virg{over} a lower--order one, by using a section of a suitable normal bundle, where \virg{over} means \virg{projecting onto}. This is the next step towards the clarification of the structure of the space of infinite--order Cauchy data.\par
The central result of the paper, Theorem \ref{thStrutturale}, can be found in the last section of theoretical character,    Section \ref{secFinale}. It implies that the pro--finite manifold $\III_{n-1}(\C)$ gives rise to three distinct secondary manifolds, one whose points are the Cauchy data, another whose points are the solutions, and the last whose points are $(n-1)$--dimensional submanifolds of solutions, thus providing a natural  common framework for these three apparently heterogeneous entities. Most importantly, it shows that the (secondary) fibers of the naturally defined maps from one secondary manifold to the other, are, in   turn, very simple secondary manifolds, namely \emph{empty equations}. Handy coordinates, indispensable for applications, are also introduced here. The final comments on Theorem \ref{thStrutturale} are collected in Section \ref{secConclusiva}, together with the  envisaged consequences and applications.\par
In the last Section \ref{SecEsempioVariazionale}, we present a variational problem simultaneously  involving Lagrangians with  $n$ and $n-1$ independent variables, and we test it on the toy model given by a simple 1--dimensional variational problem with constrained endpoints.

\subsubsection*{Notations and conventions}

Even if  we did our best to avoid proprietary notations,   we must warn the reader about  a somewhat extreme \virg{slang} we are going to use throughout this paper, in line with the most recent works on the subject (see, e.g., \cite{Luca}).  \virg{$P$ is an $M$--module} means that $P$ is the module of sections of  a bundle $\pi$ over $M$. Then the meaning of expressions like    $P\otimes_M Q$ and $\Hom_M(P,Q)$ is clear. We use both $\pi_x$ and $E_x$ as synonymous of $\pi^{-1}(x)$, where $\pi$ is a fibration of $E$. By a \virg{plane in $V$} we mean just a vector subspace of $V$; similarly, a \virg{plane in $E$}  is a subspace of some tangent space to the manifold $E$. Term \virg{space} without modifier (like the one appearing in the title of this paper) always means \virg{secondary manifold}.  \par
We use the term  \virg{leaf} for an $n$--dimensional integral submanifold of the Cartan distribution on $J^\infty(E,n)$, and we denote it by $L$. A codimension--one submanifold of a leaf is called a \virg{small leaf}, and denoted by $\Sigma$. The projection of any object $\mathcal{O}$ associated with $J^k$  (with $k=0,\ldots,\infty$) on a lower oder jet $J^l$ is denoted by $\mathcal{O}_l$; for instance, $L_0$ is the submanifold of $E$ which corresponds to the leaf $L$ (but, in this case, we  even skip the index \virg{$0$}). We allow $l$ to take the value -1, assuming that $E_{-1}$ is an arbitrary choice of the manifold of independent variables (in which case   we are considering a so--called \emph{affine chart} in $J^\infty(E,n)$), and we write \virg{$\,\loc\,$} when an equality holds in coordinates or affine charts, like, e.g., $J^\infty(E,n)\loc J^\infty(\pi)$. If $\theta\in J^\infty$ is the jet of a section in some point, then $\theta_{-1}$ is precisely that point. We use the word \virg{over} to indicate that one thing projects over another. \par
$TE$ denotes the tangent bundle of $E$, and $f_\ast$ denotes the differential of $f:E\to E'$. If $E$ is fibered, $VE$ is the vertical tangent bundle ($VJ^k$ means \virg{vertical with respect to $\pi_{k,k-1}$}).
The $\R$--dual of a vector space is denoted by $V^\vee$, and the annihilator of a subspace $W$ by $W^\dag$. The same symbol  {$P^\dag$} is used, in different contexts,  for the adjoint module to $P$.
 We prefer to say that \virg{$L$ is a leaf of $\E^{(\infty)}$}, rather than  \virg{$L_k$ is a solution of $\E\subseteq J^k(E,n)$}.\par
The $R$--distribution on $J^k$ (see \cite{Sym}) is called $R^k$,   Cartan distribution on $J^k$ is denoted by $C^k$, while that  on $J^\infty$ simply $\C$.\par

Modifier \virg{local} in front of \virg{section} or \virg{coordinates} is skipped as a rule. All PDEs are assumed to be formally integrable.\par
 
%
Greek and latin indexes range in disjoint sets: this means that the sets of coefficients $\{\omega_a\}$ and $\{\omega_\alpha\}$ cannot have any element in common. Latin indexes correspond to independent variables, and greek ones to dependent variables.
Derivations are denoted by a semicolon: $\omega_{a;\alpha}$ means $\frac{\partial \omega_a}{\partial u^{\alpha}}$.
For iterated jet spaces we encapsulate into parentheses the inner jet variables, before taking the outer derivatives, like in 
$(u_a^\alpha)_b^\beta$, or $(u)_a$. Concerning multi--indices for partial derivatives, uppercase latin letters will always denote elements of the abelian group $\N_0^{n-1}$, even if we use multiplicative notation for its operation and the symbol $\O$   for its zero    (as in the  \virg{monoidal notation}, see \cite{Luca}); the pair $(A,l)$, where $l\in\N_0$, is an element of $\N_0^{n}$, namely the one having the first $n-1$ entries in common with $A$, and the last one equal to $l$ (hence, $|(A,l)|=|A|+l$). For example,
\begin{equation}
\frac{\partial^{|A|+l}}{\partial x^{A,l}}=\frac{\partial^{i_1+\cdots+i_{n-1}+l}}{\partial (x^1)^{i_1}\cdots\partial (x^{n-1})^{i_{n-1}}\partial (x^n)^l},\quad A=(i_1,\ldots,i_{n-1})\in\N_0^{n-1},l\in\N_0.
\end{equation}

Number $n$ is fixed throughout this paper,    index $\alpha$  is always assumed to be ranging in $1,\ldots,m$, and  index $a$ in $1,\ldots,n-1$. Symbol $Aa$ represents the multi--index $A$ whose $a^\textrm{th}$ entry has been increased by one.\par
All constructions are coordinate--free. Nonetheless, many concepts look more familiar when written down in coordinates, so the reader will find several Remarks labeled \virg{Coordinates} after any intrinsic definition.\par
We have chosen the notation $J^k(E,n)$ for the     space of $k$--jets of $n$--dimensional submanifolds of $E$, just to stress the analogy with the linear case, when   one works with Grassmann manifolds $\Gr(V,n)$   instead. Alternatively, one may regard $J^k(E,n)$ as a sub--quotient of $J^k(\R^n,E)$, the space of $k$--jets of smooth maps from $\R^n$ to $E$  \emph{\`a la} P. Michor \cite{MR583436}, namely the space of $k$--jets of embeddings, factorized by the group of diffeomorphisms of   $\R^n$. Or, in a  more \virg{mechanical} perspective, $J^k(E,n)$ may be seen as the  space of $k$--jets of regular, parameter--free $n$--velocities in $E$ \emph{\`a la} D. Krupka \cite{MR0394755}. No matter which point of view is adopted, the definition is the same, viz.,
\begin{equation*}
J^k(E,n):=\coprod_{y\in E}J^k_y(E,n),\end{equation*}\begin{equation*} J_y^k(E,n):=\frac{\{L \mid L\subseteq E\textrm{ is $n$--dimensional submanifold and }L\ni y\}}{\sim_y^k},
\end{equation*}
where $\sim_y^k$ is the equivalence relation
\begin{equation*}
L_1\sim_y^k L_2\Leftrightarrow L_1\textrm{ is tangent to }L_2\textrm{ at $y$ with order }k.
\end{equation*}
The equivalence class of $L$ w.r.t. $\sim_y^k$ is denoted by $[L]_y^k$. If $E=\{(x^i,u^\alpha)\}$, and $L=\textrm{graph\,}(s)$, where $s=(s^1,\ldots,s^m)$, with $s^\alpha=s^\alpha(x^1,\ldots, x^n)$, then the jet coordinate $u_{i_1\cdots   i_n}^\alpha$ is defined as
\begin{equation*}
u_{i_1\cdots  i_n}^\alpha([L]_{s({\boldsymbol x})}^k):=\frac{\partial^{i_1+\cdots+i_{n } } s}{\partial (x^1)^{i_1}\cdots\partial (x^{n})^{i_{n}} }({\boldsymbol x}), \quad {\boldsymbol x}=(x^1,\ldots,x^n).
\end{equation*}

It is worth recalling  that the $R$--distribution is nothing but the jet--theoretic incarnation of the tautological (or universal) bundle associated to a Grassmann manifold; it associates with a point $\theta\in J^k$ the $n$--dimensional subspace of $T_{\theta_{k-1}}J^{k-1}$  spanned by
\begin{equation}\label{eqPerReferee2}
\left.\partial_l\right|_{\theta_{k-1}}+\sum_{i_1+i_2+\cdots+i_n\leq k-1}u_{i_1\cdots i_l+1\cdots i_n}^\alpha(\theta)\left.\partial_{u^\alpha_{i_1i_2\cdots i_n}}\right|_{\theta_{k-1}}, \quad l=1,\ldots,n,
\end{equation}
where $\partial_l$ is a short for $\partial_{x^l}$. 
The $R$--distribution \virg{generates} the Cartan distribution, in the sense that $\C^k_\theta=\pi_{k,k-1\ast}^{-1}(R^k_\theta)$; as such, besides the $n$ vectors \REF{eqPerReferee2}, it takes also all the $\pi_{k,k-1}$--vertical vectors to span it.

\section{Jet maps}\label{secFunctJey}
Obvious Definition \ref{defJetMap} below is given just   to simplify   subsequent  constructions. 
Let $f:E\to E'$ be a smooth map. 
\begin{proposition}
 The subset $\J^k(E,n)\df\{\theta\in J^k(E,n)\mid R_{\theta_1}\cap\ker f_\ast =0\} $ is an open   sub--bundle.  Moreover, the natural map $f_\ast: \J^k(E,n)\to J^k(E',n)$ is smooth.
\end{proposition}
\begin{proof}
 Notice  that $R_{\theta_1}\cap\ker f_\ast$ is the kernel  of the restriction
  \begin{equation}\label{eqDelReferee}
 f_\ast|_{R_{\theta_1}}: R_{\theta_1}\longrightarrow T_{f(\theta_0)}E'
\end{equation}
 of $f_\ast$ to $R_{\theta_1}\subseteq T_{\theta_0}E$. In turn, $R_{\theta_1}=\Span{\left.\partial_i\right|_{\theta_0}+u_i^\alpha(\theta)\left.\partial_{u^\alpha}\right|_{\theta_0}\mid i=1,\ldots,n}$, where $\{x^i,u^\alpha\}$ are local coordinates on $E$, and $\partial_i$ is a short for $\partial_{x^i}$. Hence, $\theta\in J^k(E,n)$ belongs to $\J^k(E,n)$ if and only if \REF{eqDelReferee} is injective, i.e., if and only if the $n$ tangent vectors $f_\ast\left(\left.\partial_i\right|_{\theta_0}+u_i^\alpha(\theta)\left.\partial_{u^\alpha}\right|_{\theta_0}\right)$  are linearly independent in $T_{f(\theta_0)}E'$, which means that the $n$--multivector 
   \begin{equation}
\boldsymbol{Y}(\theta):=f_\ast\left(\left.\partial_1\right|_{\theta_0}+u_1^\alpha(\theta)\left.\partial_{u^\alpha}\right|_{\theta_0}\right)\wedge\cdots\wedge f_\ast\left(\left.\partial_ni\right|_{\theta_0}+u_n^\alpha(\theta)\left.\partial_{u^\alpha}\right|_{\theta_0}\right)  \in T_{f(\theta_0)}^{\wedge n}E'
\end{equation}
must be nonzero. The result follows from the fact that $\boldsymbol{Y}$ depends smoothly on $\theta$, and that $\boldsymbol{Y}(\theta)\neq 0$ is an open condition.
\end{proof}
\begin{definition}\label{defJetMap}
 We call $\J^k(E,n)$ the bundle of \emph{$f$--mappable jets}, and $f_\ast$ the induced \emph{jet map}. 
\end{definition}
\begin{example}\label{exGettiDiSezioniComeGettiMappabili}
 If $\pi$ is a bundle of $E$   over $E_{-1}$, then $\pi$--mappable jets are just jets of sections of $\pi$, i.e., $\J^k(E,n)=J^k(\pi)$. In this case, the map $\pi_\ast$ is not very interesting, since $J^k(E_{-1},n)$ is a one--point manifold.
\end{example}
\begin{example}\label{exCarrierJet}
 If $f$ is an embedding, then all jets are $f$--mappable. In particular, if $s$ is a section of $\pi$, then  all jets in $J^l(E_{-1},r)$ are  $j_k(s)$--mappable, so that there are well--defined smooth maps
 \begin{equation}
j_k(s)_\ast : J^l(E_{-1},r) \longrightarrow J^l(J^k(E,n),r).
\end{equation}
Similarly, for any $n$--dimensional submanifold $L\subset E$, and $r\leq n$, there is a well-defined smooth map
\begin{equation}\label{eqJetCarrier}
j_k(L)_\ast : J^l(L,r) \longrightarrow J^l(J^k(E,n),r).
\end{equation}
 Map \REF{eqJetCarrier} is the key to \virg{lift} an $r$--dimensional submanifold of $E$ to a special submanifold of $J^k(E,n)$, namely an \emph{involutive} one (Section \ref{secInvPlanes}). 
  \end{example}

\begin{remark}\label{remarkCheQuiNonCentraNulla}
It should be stressed that, as a rule, there is no natural embedding
\begin{equation}\label{eqEmbeddingCheNonEsiste}\xymatrix{
 J^k(E,n-1)\ar@{^(->}[r]& J^k(E,n) 
 }
\end{equation}
and \REF{eqJetCarrier} has to be regarded as the closest way one has to \REF{eqEmbeddingCheNonEsiste}, when   the necessity arises   to force jets of $(n-1)$--dimensional submanifolds into the jet bundle of $n$--dimensional submanifolds. Nonetheless, \REF{eqEmbeddingCheNonEsiste} can be accomplished in a local, non--canonical way. Namely,  equip  $E_{-1}$ with a metric $g$. Then each small submanifold $\Sigma\subseteq E$ can be seen as the graph of a section $\sigma$ of $\pi$ over $\Sigma_{-1}$. So, $[\Sigma]_{y}^k=[\sigma]^k_{y_{-1}}$, and  $\sigma$ can be extended to a constant section $\iota_{y_{-1}}({\sigma})$ along the orthogonal direction to $T_{y-1}\Sigma_{-1}$. Then \REF{eqEmbeddingCheNonEsiste} si given by  $[\Sigma]_{y}^k\mapsto [\iota_{y_{-1}}({\sigma})]_{y_{-1}}^k$.\par
\end{remark}
\begin{example}\label{exCarrierJetGASANTE}
 All elements of $J^k(E,n)$, seen as a subset of $J^1(J^{k-1}(E,n),n)$,  are $\pi_{k-1,k-2}$--mappable, and $(\pi_{k-1,k-2})_\ast=\pi_{k,k-1}$.
 \end{example}
\begin{example}\label{exJetProfiniti}
Let $E=\lim E_k$ be a pro--finite manifold. Then $J^r(E,n):=$\linebreak $\lim\J^r(E_k,r)$.
 \end{example}

 \section{Grassmannians and flag manifolds}\label{secFlagGr}
 The following   basic facts about flags manifolds and Grassmannians belong to the common knowledge, so that  it is hard to point an appropriate reference. Concerning  the link between Grassmannians and jet spaces, a nice exposition can be found in the classical book \cite{Bryant}. \par

\begin{definition}
 $\Gr(V,n)\df\{L\subseteq V\mid L\textrm{ is a $n$--dimensional plane in }V\}$ is the Grassmannian of $n$--dimensional planes in $V$
\end{definition}
Recall that over  $\Gr(V,n)$ it grows the so--called \emph{universal sequence}
  of vector bundles
\begin{equation}\label{eqGrass1}
\xymatrix{
R(V,n)\ar@{^(->}[r]  \ar[dr]_R &   \Gr(V,n)\times V \ar[d]^\tau  \ar@{->>}[r] &    N(V,n) \ar[dl]^N\\
 & \Gr(V,n),
}
\end{equation}
where
  $\tau$ is the trivial   bundle,  $R$ is the   {\it tautological   bundle}, and  $N$ is the {\it  normal bundle}. 
By definition, $R_L=L$ (hence the name \virg{tautological}) and $N_L=\frac{V}{L}$, for all $L\in\Gr(V,n)$. In particular, $\rank R=n$ and $\rank N=\dim V-n$, and it  holds the   non--canonical bundle isomorphism
\begin{equation}\label{eqFondamentaleTangenteGrassmann}
\tau_{ \Gr(V,n) }\cong \Hom(R , N )=R^\vee\otimes_{\Gr(V,n) } N,
\end{equation}
incidentally showing that $\dim \Gr(V,n)=(\dim V-n)n$.
Let now  $\xi: {E}\to M$  be  
  a vector bundle.   
  
\begin{lemma}
A smooth bundle  $ \Gr( {E},n)$ over $M$ exists,  and a short exact sequence of vector bundles over $ \Gr( {E},n)$,
\begin{equation}\label{eqGrass2}
\xymatrix{
R( {E},n)\ar@{^(->}[r]  \ar[dr] &   \Gr( {E},n)\times_M { {E}} \ar[d]  \ar@{->>}[r] &    N( {E},n) \ar[dl]\\
 & \Gr( {E},n), }
\end{equation}
such that $\Gr(E,n)_x=\Gr(E_x,n)$ and the restriction of \REF{eqGrass2} to a point $x\in M$ equals \REF{eqGrass1}, with $V=E_x$, and the following bundle isomorphism holds:
\begin{equation}\label{eqFondamentaleTangenteGrassmannVert}
V \Gr  (E,n)\cong \Hom_{ \Gr  (E,n)}(R(E,n) , N(E,n))=R(E,n)^\vee\otimes_{ \Gr  (E,n)}N(E,n). 
\end{equation}

\end{lemma}
\begin{proof}
The first statement follows straightforwardly (by using transition functions) from the fact that the universal sequence \REF{eqGrass1} is well--behaved w.r.t. linear transformations of $V$, i.e., each $\phi\in\GL(V)$ induces a diffeomorphism $\underline{\phi}$ of $\Gr(V,n)$, and bundle automorphisms of $R(V,n)$, $\Gr(V,n)\times V$, and $N(V,n)$, which cover $\underline{\phi}$. \par

 The second one is a consequence of \REF{eqFondamentaleTangenteGrassmann}, since $V_x\Gr(E,n)$ coincides with\linebreak $T \Gr(E_x,n)$.
\end{proof}

\begin{example}
 $\Gr(TE,n)$ is one possible definition of $J^1(E,n)$ (see, e.g, \cite{Bryant}). An alternative one, given in term of tangency classes, can be found, e.g., in \cite{Sym}.
\end{example}

\begin{example}[Definition of flag manifolds]\label{exFlag}
 $\Gr(R(V,n),n-1)$ is the flag manifold $\Gr(V,n,n-1)$. The corresponding canonical sequence
 \begin{equation}\label{eqGrass22}
\xymatrix{
R(R(V,n),n-1)\ar@{^(->}[r]  \ar[dr] &   \Gr(R(V,n),n-1)\times_{\Gr(V,n)} { R(V,n)} \ar[d]  \ar@{->>}[r] &    N( R(V,n),n-1) \ar[dl]\\
 & \Gr( R(V,n),n-1) }
\end{equation}
is simply denoted by
 \begin{equation}\label{eqGrass222}
\xymatrix{
r\ar@{^(->}[r]  \ar[dr] &  R \ar[d]  \ar@{->>}[r] &   n\ar[dl]\\
 & \Gr( V,n,n-1). }
\end{equation}
By definition, if $\theta=(L,\Sigma)\in\Gr(V,n,n-1)$, $r_\theta=\Sigma$, $R_\theta=L$, and $n_\theta=\frac{L}{\Sigma}$.
\end{example}

Example \ref{exFlag} shows that $\Gr(V,n,n-1)$ is naturally fibered over $\Gr(V,n)$, and that  $\Gr(V,n,n-1)_L=\Gr(L,n-1)$, for all $L\in \Gr(V,n)$. 
\begin{fact}[Canonical fibrations of flag manifolds]\label{factFattoInteressante}
  $\Gr(V,n,n-1)$ is naturally fibered over both $\Gr(V,n-1)$ and $\Gr(V,n)$, i.e., 
   \begin{equation}
  \xymatrix{
 & \Gr(V,n,n-1)  \ar[dr]^{n^1}\ar[dl]_{p^1}& \\
\Gr(V,n)  & & \Gr(V,n-1) 
  }
\end{equation}
where     $n^1_\Sigma=\P( {\Sigma}^\dag)$, for all $\Sigma\in \Gr(V,n-1)$ and $n^1_L=\P(L^\vee)$, for all $L\in \Gr(V,n)$.
\end{fact}

The definition given by Example \ref{exFlag}, the sequence  \REF{eqGrass222} and the fibrations \REF{factFattoInteressante} are easily generalized to flags with more indices and complete flags, but they will not play a relevant role in our analysis. The aim of this section was to stress that, even if the family of all $n$--dimensional planes in $V$ has a natural smooth manifold structure,   the same is not true if in the same family enter   $(n-1)$--dimensional planes, since, roughly speaking, the latter are more numerous than the former. Then one is forced to introduce a certain  redundancy in the information about $n$--dimensional planes, to get something smooth: the result is $\Gr(V,n,n-1)$. A   redundancy conceptually similar, but technically more involved,  will have to be introduced in the context of nonlinear PDEs, in order to treat   leaves and small leaves \virg{as members of the same family}.
 

\section{The equation of involutive planes}\label{secInvPlanes}
Let $P$ an $E$--module, and suppose that $\Delta:=\ker\Omega$ is a distribution given by means of the $P$--valued 1--form $\Omega$. Let also   $\Pi\in\Lambda^2(\Delta^\vee)\otimes_E\frac{\Delta^{(1)}}{\Delta}$ be the curvature form of $\Delta$.
\begin{definition}\label{defPianoInvolutivo}
 A tangent plane $R$ to $E$ is called \emph{involutive} if  
 \begin{equation}
\Omega|_R=0,\quad \Pi|_R=0.\label{eqEQ-fond}
\end{equation}

  The totality of $r$--dimensional involutive   planes of $E$ is the \emph{equation of involutive $r$--dimensional planes} of $E$, and denoted by $\I_r(\Delta)$.
\end{definition}

 Example \ref{exCoordInvEqqq} below should convince the reader about the smoothness of the submanifold $\I_r(\Delta)$ of $J^1(E,r)$.

\begin{remark}\label{remEreditarieta}
 Definition \ref{defPianoInvolutivo} is hereditary for linear subspaces, since so are conditions \REF{eqEQ-fond}.
\end{remark}


\begin{example}[Coordinates]\label{exCoordInvEqqq}

Let $E=\{(x^i,u^\alpha)\}$, and $\Delta$ given by means of 1--forms
\begin{equation}
\Delta={{\cap}_{A\in\mathbb{A}}} \ker \omega^A,
\end{equation}
with $\omega^A=\omega^A_idx^i+\omega^A_\alpha du^\alpha$.
Then
\begin{equation}
\I_r(\Delta)=\{\theta\in J^1(E,r)\ |\ \omega^A|_{R_\theta}=0,\ d\omega^A|_{R_\theta}=0\}
\end{equation}
  is locally given by the vanishing of the functions
 \begin{eqnarray}
f_i^A & = & \omega^A_i +\omega^A_\alpha  u^\alpha_i \label{eqFunzLocDistrib1},\\
f_{ij}^A &=& \omega^A_{[i;j]} +\omega^A_{[[i;\alpha]}u_{j]}^\alpha +\omega^A_{[\alpha;\beta]}u_j^\alpha u_i^\beta.\label{eqFunzLocDistrib2}
\end{eqnarray}

\begin{remark}\label{remConseqDiff}
 Let $\E\subset J^1(E,r)$ be given just by the vanishing of \REF{eqFunzLocDistrib1} alone. Then $\I_r(\Delta)=\pi_{2,1}(\E^{(1)})$, i.e., \REF{eqFunzLocDistrib2} are differential consequences of \REF{eqFunzLocDistrib1}.
\end{remark}

\begin{example}\label{exlpRpianiSing}
 If $\C^k$ is the contact distribution on $ J^k(E,n)$, then $\I_n(\C^k)$ is the closure of $J^{k+1}(E,n)$ in $J^1(J^k(E,n),n)$.  Adherence points corresponds to the so--called singular $R$--planes (firstly studied by Vinogradov in the context of singular and multivalued solutions \cite{Vino87,Vino73}).
\end{example}

 \end{example}

 If $E$ is fibered (see Example \ref{exGettiDiSezioniComeGettiMappabili}), then  $\II_r(\Delta)\df \I_r(\Delta)\cap \J^1(E,r)$ is  an   open and dense subset of $\I_r(\Delta)$.
 
\begin{definition}
  $\II_r(\Delta)$ is the equation of \emph{horizontal} involutive $r$--dimensional planes.
\end{definition}

\begin{example}
  Let $\C^k$ be as in Example \ref{exlpRpianiSing}. Then $\II_n(\C^k)=J^{k+1}(E,n)$.
\end{example}

\begin{remark}
Leaves of $\I_r^{(\infty)}(\Delta)$ are in one--to--one correspondence with $r$--dimen\-sional involutive submanifolds of $\Delta$ (see Remark \ref{remConseqDiff}).  In other words, $\I_r^{(\infty)}(\Delta)$ is the secondary manifold whose points are the $r$--dimensional involutive submanifolds of $\Delta$.
\end{remark}
 
 \section{Flags jet bundles}\label{secFlagJet}
Remark \ref{remEreditarieta} motivates the key  definition \ref{defFlagBund} below.\par
 Let $n= n_d>n_{d-1}>\cdots>n_2>n_1>0$ be integers, and consider the fibered product
 \begin{equation}
X \df J^k(E,n)\times_{J^{k-1}(E,n)}\J^1(J^{k-1}(E,n),n_{d-1})\times \cdots\times_{J^{k-1}(E,n)}\J^1(J^{k-1}(E,n),n_{1}).
\end{equation}
A point $\Theta\in X$ can be seen as a $d$--tuple of planes  in $J^{k-1}(E,n)$, whose  dimension decreases from $n_d$ to $n_1$, only whose  first entry  is required to be involutive.

\begin{proposition}
Denote by $R^{i}_\Theta$ the $i\Th$ plane in $\Theta$, i.e., the one of dimension $n_i$. Then the subset 
\begin{equation}\label{defFlagJets}
 J^k(E,n_d,n_{d-1},\ldots,n_2,n_1)\df  \{\Theta\in X\mid R^{i}_\Theta\supseteq R^{{i-1}}_\Theta,\quad\forall i=2,\ldots,d   \}
\end{equation}
is a smooth sub--bundle of $X$. 
\end{proposition}
\begin{proof}
 Easily checked in coordinates (see Remark \ref{remCoordinateDatiFiniti} below). 
\end{proof}

\begin{definition}\label{defFlagBund}

$ J^k(E,n_d,n_{d-1},\ldots,n_2,n_1)$ defined as \REF{defFlagJets} is the $k$--order \emph{flag jet bundle} over $J^{k-1}(E,n)$.
  $J^k(E,n,n-1,\ldots,2,1)$ is the  $k$--order \emph{complete flag jet bundle}.
\end{definition}

\begin{fact}
 $J^k(E,n,n-1,\ldots,2,1)$ projects naturally over any $J^k(E, i)$. 
 
\end{fact}
%


From now on, the  focus will be on $J^k(E,n,n-1)$. An element $\Theta\in J^k(E,n,n-1)$ is written as a pair $(R_\Theta, r_\Theta)$.

\begin{remark}[Coordinates I]\label{remCoordinateDatiFiniti}
 Let $\Theta\in  J^k(E,n)\times_{J^{k-1}(E,n)}\J^1(J^{k-1}(E,n),n-1)$, and consider its coordinate expression
 $$\Theta=(x^a,t,\underset{|A|+l \leq k}{u^\alpha_{A,l}},t_a, \underset{|A'|+l' \leq k-1}{(u^\alpha_{A',l'})_a})$$ into an adapted chart. Then
 \begin{eqnarray}
R_\Theta &=&\Span{ \partial_a+u_{Aa,l}^\alpha\partial_{u_{A,l}^\alpha}\mid a=1,\ldots,n-1}+\Span{\partial_t+u_{A,l+1}^\alpha\partial_{u_{A,l}}},\label{eqPianoGrande}\\
r_\Theta &=& \Span{\partial_a+t_a\partial_t+(u_{A',l'}^\alpha)_a\partial_{u^\alpha_{A,l}}\mid a=1,\ldots,n-1},\label{eqPianoPiccolo}
\end{eqnarray}
are the corresponding planes in $J^{k-1}(E,n)$. Observe   that \REF{eqPianoGrande} contains \REF{eqPianoPiccolo} if and only if each generator of the latter is a linear combination of generators of the former, viz., 
\begin{equation}\label{eqIntermediaPerCoordinateCoppie}
\partial_a+t_a\partial_t+(u_{A,l}^\alpha)_a\partial_{u^\alpha_{A,l}} =  \partial_a+u_{Aa,l}^\alpha\partial_{u_{A,l}^\alpha} +t_a({\partial_t+u_{A,l+1}^\alpha\partial_{u_{A,l}}})\quad \forall   a ,
\end{equation}
where $|A|+l\leq k-1$. In their turn, vector equalities \REF{eqIntermediaPerCoordinateCoppie} are equivalent to the system of equations
\begin{equation}\label{eqEquazioniCoordinateDatiFinitiEQ}
(u_{A,l}^\alpha)_a=u_{Aa,l}^\alpha+t_au_{A,l+1}^\alpha,\quad |A|+l\leq k-1.
\end{equation}
Hence, 
\begin{equation}\label{eqCoordinateGiusteSuDAtiFiniti} 
x^a,t,\underset{|A|+l \leq k}{u^\alpha_{A,l}},t_a
\end{equation}
 can be assumed as coordinates on $J^k(E,n,n-1)$. A rough interpretation of \REF{eqEquazioniCoordinateDatiFinitiEQ} is the following: in   $J^k(E,n,n-1)$   the independent variable $t$ has become a dependent one, so that $u^\alpha_{A,l}$ depends on   $x^a$ not only directly (first summand in the right--hand side), but also through $t$  (second summand).
\end{remark} 

\begin{remark}[Coordinates II]\label{remCoordinateDatiFinitiAlt}

Equations \REF{eqEquazioniCoordinateDatiFinitiEQ} furnish another coordinate system on $J^k(E,n,n-1)$, which will be handier than \REF{eqCoordinateGiusteSuDAtiFiniti} in the study of normal bundles (see Remark \ref{remCoordInfCauchDat} later on), namely
\begin{equation}\label{eqCoordinateGiusteSuDAtiFinitiBIS}
x^a,t,\underset{|A|+l \leq k-1}{u^\alpha_{A,l}},u^\alpha_{\O,k},t_a,\underset{|A'|+l' = k-1}{(u^\alpha_{A',l'})_a}.
\end{equation}

\end{remark}

\begin{remark}\label{RemCoordInvol}
 Observe that, equations \REF{eqFunzLocDistrib1} coincides with \REF{eqEquazioniCoordinateDatiFinitiEQ}, for the forms
 \begin{equation}
\omega^\alpha_{A,l}\df du_{A,l}^\alpha-u_{Aa,l}^\alpha dx^a -u_{A,l+1}^\alpha dt,\quad |A|+l\leq k-2,
\end{equation}
defining $\C^{k-1}$. On the other hand,  equations \REF{eqFunzLocDistrib2}, which reads
\begin{equation}
(u^\alpha_{A[a,l})_{b]}=(u^\alpha_{A,l+1})_{[a}t_{b]},\quad |A|+l\leq k-2,
\end{equation}
are algebraic consequences of \REF{eqEquazioniCoordinateDatiFinitiEQ}. 
\end{remark}

\begin{lemma}\label{lemLocaleBellino}
Coordinates \REF{eqCoordinateGiusteSuDAtiFiniti}  represent a local diffeomorphism 

 \begin{equation}\label{eqDescrizioneLocaleDatiFiniti}
J^k(E,n,n-1)\loc J^k(E,n)\times_{E_{-1}}J^1(E_{-1},n-1).
\end{equation}
 
\end{lemma}
\begin{proof}
Let $(\theta,\theta')\in  J^k(E,n)\times_{E_{-1}}J^1(E_{-1},n-1)$, with  $\theta=[s]_x^k$, where $s$ is a   section of $\pi:E\to E_{-1}$ and  $\theta'\in J^1_x(E_{-1},n-1)$. 
Consider the jet map $j_k(s)_\ast:J^1(E_{-1},n-1)\longrightarrow J^1(J^{k-1}(E,n),n-1)$ (Definition \ref{defJetMap}). It is easy to see that $j_k(s)_\ast(\theta')$ is a small plane contained in $R_\theta$, whose definition is independent on the choice of $s$.  
%
%
%
 Correspondence \REF{eqDescrizioneLocaleDatiFiniti} is given precisely by  
 \begin{equation}\label{eqSollevamentoPianiPiccolo}
(R_\theta,j_k(s)_\ast(\theta'))\leftrightarrow (\theta,\theta').
\end{equation}
\end{proof}
Paraphrasing \REF{eqSollevamentoPianiPiccolo},   $\theta $ has been used to \virg{lift} the small plane $R_{\theta'}$ in $E_{-1}$, i.e., the $1\St$ jet of a Cauchy surface, to a small involutive (horizontal) plane in $J^{k-1}(E,n)$ (see also Example \ref{exCarrierJet}).  However, since $R_{\theta'}$ is small, for the  purpose of lifting,  it is not necessarily the whole jet $\theta=[s]_x^k$, but rather the $k-1\St$ jet of $s$, plus the $k\Th$  derivatives of $s$ along  $R_{\theta'}$. So, elements of $J^k(E,n,n-1)$ cannot yet be called $1\St$ jets of $k-1\St$ order Cauchy surfaces (see Section \ref{secCauchDat} below) since they contain extra information. As we show in Section \ref{secFlagJetInvpla} below, this extra information is discarded by the natural projection of $J^k(E,n,n-1)$ over $\II_{n-1}(\C^{k-1})$.

\begin{corollary}
 $\dim J^k(E,n,n-1)=\dim J^k(E,n,n-1)+\dim J^1 (E_{-1},n-1) - n$.
\end{corollary}
\begin{proof}
 Directly from Remark \ref{remCoordinateDatiFiniti}.
\end{proof}

\section{Flag jets and involutive planes}\label{secFlagJetInvpla}
 
The notions of a flag jet (introduced in Section \ref{secFlagJet} above) and that of an involutive plane (introduced in Section \ref{secInvPlanes} above) are tightly interrelated in view of two simple facts. The first is that an element $r \in\II_{n-1}(\C^k)$ can be seen as a \virg{relative} flag of  planes in $J^k$, in a sense elucidated by     Lemma \ref{lemmaStupido} below.
\begin{lemma}\label{lemmaStupido}
Map
\begin{eqnarray*}
\II_{n-1}(\C^k) & \stackrel{q^k}{\longrightarrow} & J^{k-1}(E,n,n-1),\\
\C^k_\theta\supseteq r &\longmapsto & (R_\theta, r_{k-1}),
\end{eqnarray*}
is a  bundle.
\end{lemma}
\begin{proof}
  By definition, $r\subseteq \C_\theta$ is horizontal, so $r_{k-1}$ is an $(n-1)$--dimensional subspace of $R_\theta$, i.e., $r$ determine the flag $(R_\theta,r_r)$ on $J^{k-1}$.  Smoothness follows from Remark \ref{remCoordQuKappa}.
\end{proof}
The second is that  a flag projects over the space of involutive small planes, as shown by Lemma \ref{lemmaStupido2} below.
 
 \begin{lemma}\label{lemmaStupido2}
The canonical bundle
\begin{equation*}
J^k(E,n)\times_{J^{k-1}(E,n)} \J^1(J^{k-1}(E,n),n-1)   {\longrightarrow}    \J^1(J^{k-1}(E,n),n-1) 
\end{equation*}
restricts to a bundle
\begin{equation}
J^k(E,n,n-1)\stackrel{n^k}{\longrightarrow} \II_{n-1}(\C^{k-1}).\label{eqProBandiereSuInvolutive}
\end{equation}

\end{lemma}
\begin{proof}
 Directly from Remark \ref{remEreditarieta}.
\end{proof}
 
 As we shall see, taking the inverse limit, \virg{relative} becomes \virg{absolute}, and the two sides of \REF{eqProBandiereSuInvolutive} will coincide (Theorem \ref{thFagSonoCoseInvolutive}). 
The key tool is provided by   diagram \REF{eqDiagrammaLibroAperto} below.
\begin{lemma}\label{lemDefProjFlag}
 The bundle 
 \begin{eqnarray}
J^k(E,n)\times_{J^{k-1}(E,n)} \J^1(J^{k-1}(E,n),n-1) & \longrightarrow & J^{k-1}(E,n)\times_{J^{k-2}(E,n)} \J^1(J^{k-2}(E,n),n-1),\nonumber\\
\Theta=(R_\Theta, r_\Theta)&\longmapsto &\Theta_{k-1}\df(( R_\Theta)_{k-1}, ( r_\Theta)_{k-1}),\nonumber
\end{eqnarray}
restricts to a bundle  
\begin{equation}
J^k(E,n,n-1)\stackrel{\pi^\textrm{\normalfont flag}_{k,k-1}}{\longrightarrow }J^{k-1}(E,n,n-1).
\end{equation}

\end{lemma}
Let $(\pi_{k-1,k-2})_\ast$ be the jet map of $\pi_{k-1,k-2}$ (see Definition \ref{defJetMap}).
\begin{lemma}\label{lemProInv}
 The bundle
 \begin{equation}
 \J^1(J^{k-1}(E,n),n-1)\stackrel{(\pi_{k-1,k-2})_\ast }{ \longrightarrow} \J^1(J^{k-2}(E,n),n-1)\nonumber
\end{equation}
restricts to a bundle
\begin{equation}
\II_{n-1}(\C^{k-1})\stackrel{{\pi^\textrm{\normalfont inv.}_{k,k-1}}}{\longrightarrow } \II_{n-1}(\C^{k-2}).
\end{equation}

\end{lemma}
\begin{remark}\label{remStupidooo}
 The bundle of  $J^k(E,n)\times_{J^{k-1}(E,n)} \J^1(J^{k-1}(E,n),n-1)$ over the first factor determines a bundle
 \begin{equation}
p^k:J^k(E,n,n-1)\longrightarrow  J^k(E,n).
\end{equation}

\end{remark}
In view of the above lemmas, it makes sense to construct below   diagram \REF{eqDiagrammaLibroAperto}, where unlabeled arrow are canonical embeddings/bundles.\par
\hspace{-2.5cm}{\parbox{15cm}{\begin{equation}\label{eqDiagrammaLibroAperto}
  \xymatrix{
 & J^k(E,n,n-1)\ar@{->>}[dd]^{\pi^\textrm{flag}_{k,k-1}} \ar[dr]^{n^k}\ar[dl]_{p^k}\ar@{^(->}[r]&J^k(E,n)\times_{J^{k-1}(E,n)} \J^1(J^{k-1}(E,n),n-1)\ar@{->>}[dr]&\\
 J^k(E,n)\ar@{->>}[dd]^{\pi_{k,k-1}} & & \II_{n-1}(\C^{k-1})\ar@{->>}[dd]^{\pi^\textrm{inv.}_{k,k-1}}\ar@{^(->}[r]\ar[dl]^{q^k}&   \J^1(J^{k-1}(E,n),n-1)\ar@{..>}[ddlll]\ar@{->>}[dd]^{(\pi_{k-1,k-2})_\ast} \\
  & J^{k-1}(E,n,n-1) \ar[dr]^{n^{k-1}}\ar[dl]_{p^{k-1}}&& \\
 J^{k-1}(E,n) & & \II_{n-1}(\C^{k-2})\ar@{^(->}[r]  & \J^1(J^{k-2}(E,n),n-1) \\
 }
\end{equation}}}
\begin{corollary}\label{corCorollarioEgregio}
Diagram \REF{eqDiagrammaLibroAperto}  
 is commutative.
\end{corollary}

\begin{proof}
 Commutativity of the rightmost square follows from Lemma \ref{lemProInv}, much as   commutativity of the upper parallelepiped is a consequence of Lemma \ref{lemmaStupido2}. Commutativity of the leftmost parallelepiped follows from Lemma \ref{lemDefProjFlag} and Example \ref{exCarrierJetGASANTE}.\par
 Commutativity of the central triangle follows from Lemma \ref{lemmaStupido} and Lemma \ref{lemDefProjFlag}, while Lemma \ref{lemProInv}, together with Lemma \ref{lemmaStupido} and Lemma \ref{lemmaStupido2}, guarantees commutativity of the lower triangle.\par
 Finally, by applying  the composition $p^{k-1}\circ q^k$ to an element $r \in\II_{n-1}(\C^k)$, in view of Lemma \ref{lemmaStupido} and Remark \ref{remStupidooo}, one gets the unique $\theta\in J^{k-1}(E,n)$ such that $r\subseteq\C^{k-1}_\theta$, i.e., $\pi_{1,0}(r)$. This proves that $p^{k-1}\circ q^k$ is the restriction of $\pi_{1,0}: \J^1(J^{k-1}(E,n),n-1)\longrightarrow J^{k-1}(E,n)$.
\end{proof}

\begin{remark}[Coordinates on $\II_{n-1}(\C^{k-1})$]
A point 
\begin{equation}
\theta=(x^a,t,\underset{|A|+l \leq k-1}{u^\alpha_{A,l}},t_a, \underset{|A'|+l' \leq k-1}{(u^\alpha_{A',l'})_a})\in J^1(J^{k-1}(E,n),n-1)
\end{equation}
determines the small plane $r_\theta$ (see \REF{eqPianoPiccolo}), which, in view of Remark \ref{RemCoordInvol}, is involutive iff \REF{eqEquazioniCoordinateDatiFinitiEQ} are satisfied. So, 
\begin{equation}\label{eqCoordPianPicInv}
x^a,t,\underset{|A|+l \leq k-1}{u^\alpha_{A,l}},t_a, \underset{|A'|+l' = k-1}{(u^\alpha_{A',l'})_a}
\end{equation}
can be taken as coordinates on $\II_{n-1}(\C^{k-1})$. Hence, by comparing \REF{eqCoordPianPicInv} with \REF{eqCoordinateGiusteSuDAtiFinitiBIS}, one sees that the  $u^\alpha_{\O,k}$'s are fiber coordinates of $n^k$.  
Observe also that 
\begin{equation}\label{eqFibraEnneInCoordinate}
n^k_{\theta}=\{\Theta\in J^k(E,n)\mid R_\Theta\supset r_\theta\}
\end{equation}
is a subset of  $\pi_{k,k-1}^{-1}(\theta_0)$, which is parametrized by top derivatives, i.e. coordinates 
\begin{equation}\label{eqCoordFibraJet}
u^\alpha_{A,l},\quad |A|+l=k,
\end{equation}
and, thanks again to \REF{eqEquazioniCoordinateDatiFinitiEQ}, any $u^\alpha_{A,l}$ in \REF{eqCoordFibraJet} with $A\neq\O$ can be expressed in terms of $u^\alpha_{A',l+1}$, with $|A'|=|A|-1$, so that the $u^\alpha_{\O,k}$'s must be coordinates along $n^k$. 
\end{remark}

\begin{remark}[Fiber coordinates of $q^k$]\label{remCoordQuKappa}
Comparing \REF{eqCoordPianPicInv} with \REF{eqCoordinateGiusteSuDAtiFiniti}, one sees that 

\begin{equation}
(u_{A,l})_a^\alpha,\quad |A|+l=k-1,
\end{equation}
are coordinates along the fibers of $q^k$
 
\end{remark}

An easy consequence of Corollary \ref{corCorollarioEgregio} is Theorem \ref{thFagSonoCoseInvolutive} below, which establishes that the tower of flag jets carries the same information as the tower of equations of involutive small planes, i.e.,   that the tangent space to a small leaf of $J^\infty(E,n)$ is the same as a flag.

\begin{theorem}\label{thFagSonoCoseInvolutive}
\begin{equation}
\lim_{\leftarrow}\pi^{\textrm{\normalfont flag}}_{k,k-1}\cong \I_{n-1}(\C).
\end{equation}
\end{theorem}
\section{Normal bundles and finite--order Cauchy data}
Importance of Theorem  \ref{thFagSonoCoseInvolutive}  for the geometrical theory of Cauchy data is twofold. First, it shows that the  flag jet construction is just an alternative description of the $1\St$ order nonlinear PDE $\I_{n-1}(\C)$, which    can be discarded if one is merely interested in the  secondary manifold $\I^{(\infty)}_{n-1}(\C)$ (i.e., the space of Cauchy data according to Definition \ref{defSpazioDatiCauchy} below). One the other hand, $\I^{(\infty)}_{n-1}(\C)$ would be rather difficult to work with, without some important insights on its structure  (see Theorem \ref{thStrutturale} below), which follows from    the factorization
\begin{equation}\label{eqFattorizzazioneDerivate}
\pi^{\textrm{\normalfont flag}}_{k,k-1}=q^k\circ n^k.
\end{equation}
\begin{remark}
 Since (Corollary \ref{corCorollarioEgregio}) coordinates along $\pi^{\textrm{\normalfont flag}}_{k,k-1}$ are the same as those along $\pi_{k,k-1}$, i.e., $u^\alpha_{A,l}$, with $|A|+l=k$, a suggestive paraphrase of   \REF{eqFattorizzazioneDerivate} is that it allows to add to the coordinates of $J^{k-1}(E,n)$, separately, first the $k\Th$ derivatives with at least one internal direction (i.e., the coordinates $(u_{A,l})_a^\alpha$, $ |A|+l=k-1$, along $q^k$) and, then, the $k\Th$ purely normal derivative (i.e., the coordinates $u^\alpha_{\O,k}$ along $n^k$), thus obtaining $J^{k}(E,n)$. However, \virg{internal} and \virg{normal}, are to be understood in a \emph{universal} sense, that it, valid for any given (infinitesimal) \virg{space--time splitting} of the independent variables, i.e., a point of    $J^1(E_{-1},n-1)$ (Lemma \ref{lemLocaleBellino}). This is why (locally), the above passage from  the bundle $J^{k-1}(E,n)$ to  the bundle $J^{k}(E,n)$ is valid only if we fiber--multiply them by $J^1(E_{-1},n-1)$.
\end{remark}

\begin{remark}\label{remNormBundleUno}
 For $k=1$, diagram \REF{eqDiagrammaLibroAperto} yields
 \begin{equation}
  \xymatrix{
 & J^1(E,n,n-1)  \ar[dr]^{n^1}\ar[dl]_{p^1}& \\
 J^1(E,n)  & &  J^1(E,n-1), 
  }
\end{equation}
i.e., 
 \begin{equation}
  \xymatrix{
 & \Gr(TE,n,n-1)  \ar[dr]^{n^1}\ar[dl]_{p^1}& \\
\Gr(TE,n)  & & \Gr(TE,n-1) ,
  }
\end{equation}
so that $p^1_R=\P(R^\vee)$ and $n^1_r=\P(r^\dag)$ (see Fact \ref{factFattoInteressante}) for any $R\in  J^1(E,n) $  and $r\in  J^1(E,n-1) $. In this sense,  \REF{eqDiagrammaLibroAperto}  is but a generalization of the canonical double fibered structure of flag manifolds. 
\end{remark}
\begin{definition}
 $n^k$ is the $k\Th$ \emph{normal bundle}.
 \end{definition}
 
 Unlike $n^1$, which is a smooth bundle with abstract fiber $\R\P^m$ (see Remark \ref{remNormBundleUno}), $n^k$, with $k\geq 2$ are affine bundle, of dimension $m$. Indeed, \REF{eqFibraEnneInCoordinate} can be made more precise.
\begin{corollary}\label{corSimpaticoCheLegaVariFibratiNormali}
Let $r\subseteq\C^{k-1}_\theta$ be a point of $\II_{n-1}(\C^{k-1})$. Then $n^k_r$ is an affine subspace of $V_\theta J^{k-1}$ modeled by $S^{k-1}((\frac{R_\theta}{(r)_{k-2}})^\circ)\otimes_\R n^1_{r_0}$.
\end{corollary}
 
 Corollary \ref{corSimpaticoCheLegaVariFibratiNormali} provides a link between the \emph{true} normal bundle $n^1$, i.e., the one which formalizes the idea of the normal derivative to an embedded manifold, and the higher--oder ones. This raises the possibility to join together all normal bundles. Indeed,  in virtue of   Theorem \ref{thFagSonoCoseInvolutive}, $ \I_{n-1}(\C) $ inherits a   canonical sequence, to be thought of as   the limit of sequences \REF{eqGrass222}  over finite--order flag jet bundles,   
  \begin{equation}
\xymatrix{
r\ar@{^(->}[r] \ar[dr]& \C\ar[d] \ar@{->>}[r] & \frac{\C}{r}\ar[dl]\\
& \I_{n-1}(\C), &
}
\end{equation}
and $n^1$ can be considered as a bundle over  $ \I_{n-1}(\C) $ since the latter is, in   turn, a bundle over $J^1(E,n-1)$.
\begin{definition}
\begin{equation}
 \overrightarrow{n}\df\prod_{k\in\N_0} S^k\left(\frac{\C}{r}\right)^\circ\otimes n_1
\end{equation}

 is the \emph{bundle of (infinite--order) normal directions}.
\end{definition}
\begin{remark}\label{remAffinitaFibratoNromale}
Let  $r\subseteq\C^{}_\theta$ be a point of $\I_{n-1}(\C)$. Then the $k\Th$ homogeneous component of $\overrightarrow{n}_r$, denoted by $\overrightarrow{n}_r^k$, is precisely the linear space 
over which  $n^{k+1}_{r_k}$ is modeled. 
%
\end{remark}
 
\section{Finite and infinite order Cauchy data}\label{secCauchDat}
We show now that the familiar definition of a Cauchy datum of order $k$ is naturally framed in the diagram \REF{eqDiagrammaLibroAperto}.  To begin with, let us call a small submanifold $\Sigma\subseteq E$  a \emph{Cauchy value}, or a \emph{$0\Th$ order Cauchy datum}. The reason is obvious: $\Sigma_{-1}$ is a \emph{Cauchy surface} in the manifold $E_{-1}$ of independent variables, and $\Sigma$ may be (locally) thougth of as the graph (i.e., the set of values) of a ($\R^m$--valued) function on $\Sigma_{-1}$. Then, it is natural to give the next Definition \ref{defCauhcDat}, for $k=0,1,\ldots,\infty$.

\begin{definition}\label{defCauhcDat}
 A small involutive submanifold $\Sigma\subseteq J^k( E,n)$ is called a \emph{$k\Th$ order Cauchy datum} (or, simply, a \emph{Cauchy datum}, if $k=\infty$). $\Sigma_0$ is the  Cauchy value corresponding to $\Sigma$ and $\Sigma_{-1}$, if any, is the corresponding Cauchy surface.
\end{definition}

We introduce now a secondary manifold whose secondary points (i.e., leaves) are in a natural one--to--one correspondence with small leaves of $J^{k}(E,n)$.
\begin{definition}\label{defSpazioDatiCauchy}
 $\III_{n-1}(\C^k)$ is the \emph{space of $k\Th$ order Cauchy data}. When $k=\infty$, we obtain  $\III_{n-1}(\C)$, simply called   the \emph{space of  Cauchy data}.
\end{definition}
Observe that, thanks to jet projections,  a Cauchy datum $\Sigma$ determines a tower    of $k\Th$ order Cauchy data $\Sigma_k$:
\begin{equation}\label{eqTorreDatiCauchyPRE}
\Sigma_{-1}\leftarrow\Sigma_{0}\leftarrow\cdots\leftarrow\Sigma_{k-1}\leftarrow\Sigma_{k}\leftarrow\cdots\leftarrow\Sigma.
\end{equation}
Since terms of \REF{eqTorreDatiCauchyPRE} project diffeomorphically one onto the other, so do the terms of the sequence  \REF{eqTorreDatiCauchy} below:
\begin{equation}\label{eqTorreDatiCauchy}
(\Sigma_{-1})_{(1)}\leftarrow(\Sigma_{0})_{(1)}\leftarrow\cdots\leftarrow(\Sigma_{k-1})_{(1)}\leftarrow(\Sigma_{k})_{(1)}\leftarrow\cdots\leftarrow\Sigma_{(1)}.
\end{equation}
On the other hand, the first prolongation $\Sigma_{(1)}$ is a small submanifold in $J^\infty(E,n,n-1)$, and, thanks to flag--jet projections (see Lemma \ref{lemDefProjFlag}), it determines a tower
\begin{equation}\label{eqTorreDatiCauchyProl}
(\Sigma_{(1)})_1\leftarrow(\Sigma_{(1)})_2\leftarrow\cdots\leftarrow(\Sigma_{(1)})_{k-1}\leftarrow(\Sigma_{(1)})_{k}\leftarrow\cdots\leftarrow\Sigma_{(1)},
\end{equation}
where
\begin{equation}
(\Sigma_{(1)})_k := \pi^{\textrm{flag}}_{\infty,k} (\Sigma_{(1)}).
\end{equation}
Again, terms of \REF{eqTorreDatiCauchyProl} project diffeomorphically one onto the other. Lemma \ref{lemDatiSuperioriSonoGrafici} below clarifies the relationship between the two towers, \REF{eqTorreDatiCauchy}  and  \REF{eqTorreDatiCauchyProl}, having the common inverse limit $\Sigma_{(1)}$.

\begin{lemma}\label{lemDatiSuperioriSonoGrafici}
 $(\Sigma_{(1)})_k $ is the graph of a section of $n^k$ over $(\Sigma_{k-1})_{(1)}$.
\end{lemma}

\begin{proof}
If $\Sigma_{(1)}=\{(\C_\theta,T_\theta \Sigma)\mid \theta\in\Sigma)\}$, then  
\begin{equation}
(\Sigma_{k-1})_{(1)}=\{ T_{\theta_{k-1}} \Sigma_{{k-1}}\mid \theta_{k-1}\in\Sigma_{k-1})\},
\end{equation}
while
\begin{equation}\label{eqPuntiGiu}
(\Sigma_{(1)})_k =\{(R_{\theta_k},T_{\theta_{k-1}}\Sigma_{k-1})\mid \theta_{k-1}\in\Sigma_{k-1}\}.
\end{equation}
It remains to be noticed that, in \REF{eqPuntiGiu}, $R_{\theta_k}$ belongs to the fiber of $n^k$ over ${T_{\theta_{k-1}}\Sigma_{k-1}}$.
\end{proof}
 
 Lemma \ref{lemDatiSuperioriSonoGrafici} shows that   a section $\nu^k$ of $n^k$ is the only  \virg{additional  information} needed to produce a $k\Th$ oder Cauchy datum out of a $k-1\St$ order one. Schematically,
 \begin{equation}\label{eqCretinaMaBella}
\Sigma_{k-1}\mapsto (\Sigma_{k-1})_{(1)} \mapsto \nu^k((\Sigma_{k-1})_{(1)})\mapsto p^k(\nu^k((\Sigma_{k-1})_{(1)})).
\end{equation}

\begin{fact}\label{propOriginalePenso}
 Fix a  section $\nu^k$ for any $n^k$.    Then for any Cauchy value $\Sigma_0 $ there is a unique Cauchy datum  $\Sigma$ over $\Sigma_0$ such that 
  \begin{equation}\label{eqCatenaDelCazzo}
(\Sigma_{(1)})_k=\textrm{\normalfont graph}\,(\nu^k|_{(\Sigma_{k-1})_{(1)}}),\quad k=1,2,\ldots.
\end{equation}

\end{fact}

\begin{proof} 
 Inductively make use of \REF{eqCretinaMaBella}.
\end{proof}


\begin{remark}
A valuable generalization of  Fact \ref{propOriginalePenso} would be that the space of Cauchy data over a given Cauchy value is the same as the space of sections of $\overrightarrow{n}$. This cannot be achieved, since higher--order   homogeneous components of $\overrightarrow{n}$ cannot be defined without the knowledge of  lower--order Cauchy data. However, much as affine spaces are modeled by linear ones, the space of Cauchy data we are interested in can be \virg{modeled} by  the space of sections of $\overrightarrow{n}$, in a sense clarified by   Proposition \ref{propAffineCheChissaSeFunziona}  below, which is   fundamental to prove the Structural Theorem \ref{thStrutturale}.
 
\end{remark}

\begin{proposition}\label{propAffineCheChissaSeFunziona}
 Let $\Sigma_0$ be a Cauchy value, and fix a Cauchy datum $\Sigma$ over it.  Then sections of  $ \overrightarrow{n}|_{\Sigma_{(1)}}$ are in one--to--one correspondence with  Cauchy data over $\Sigma_0$.
\end{proposition}

\begin{proof}
First of all, thanks to the chain of diffeomorphisms \REF{eqTorreDatiCauchy}, the bundle  $ \overrightarrow{n}|_{\Sigma_{(1)}}$ can be identified with  $ \overrightarrow{n}|_{(\Sigma_k)_{(1)}}$, for any $k$. Hence, a section $v$ of $ \overrightarrow{n}|_{\Sigma_{(1)}}$ may be thought of as a family $\{\nu^k\}_{k\in\N}$, where $\nu^k$ is a section of  $\overrightarrow{n}^k|_{(\Sigma_{k-1})_{(1)}}$.\par

Because of Lemma \ref{lemDatiSuperioriSonoGrafici},  
 $(\Sigma_{(1)})_1 $ is the graph of a section $\sigma^1$ of $n^1$ over $(\Sigma_{0})_{(1)}$, so, in view of Remark \ref{remAffinitaFibratoNromale}, it makes sense to define
 \begin{equation}\label{eqIndBasisGasante}
\Sigma'_1:=\textrm{graph}\,( \sigma^1 + \nu^0).
\end{equation}  
Now \REF{eqIndBasisGasante} can be used as the induction basis to subsequently \virg{adjust} the given Cauchy datum by means of the sections of $ \overrightarrow{n}$ (much as in the proof of Fact \ref{propOriginalePenso}). Indeed (see  also \REF{eqCatenaDelCazzo}), $(p^1(\Sigma'_1))_{(1)}$ is a small submanifold over $(\Sigma_0)_{(1)}$, and as such identifies diffeomorphically with $(\Sigma_1)_{(1)}$. Hence, $\nu^1$ can be understood as a bundle over  $(p^1(\Sigma'_1))_{(1)}$, and \REF{eqIndBasisGasante}  can be used again to define $\Sigma'_2$. Continuing the iteration, one defines the Cauchy datum $\Sigma'$.
%
%
%
\end{proof}

\section{The space of infinite--order Cauchy data}\label{secFinale}
Let 
\begin{equation}
\iota:{\III_{n-1}(\C)}\longrightarrow J^\infty(J^\infty(E,n),n-1)
\end{equation}

be the canonical inclusion, and
\begin{eqnarray*}
 p \df \pi_{\infty,0}|_{\III_{n-1}(\C)},&\textrm{where}&\pi_{\infty,0}:J^\infty(J^\infty(E,n),n-1)\longrightarrow J^\infty(E,n), \\
 n \df (\pi_{\infty,0})_\ast|_{\III_{n-1}(\C)},&\textrm{where}&\pi_{\infty,0}: J^\infty(E,n) \longrightarrow  E 
 \end{eqnarray*}
(see Definition \ref{defJetMap} for the meaning of $ (\pi_{\infty,0})_\ast$).
Maps $\iota$, $p$ and $n$ are conveniently depicted in the star--shaped diagram \REF{eqStrutturaDiffInData} below.
 \begin{equation}\label{eqStrutturaDiffInData}
\xymatrix{
& J^\infty(J^\infty(E,n),n-1)& \\
& \III_{n-1}(\C)\ar[dl]_{p}\ar[dr]^{n}\ar@{_(->}[u]^{\iota}& \\
J^\infty(E,n)& & J^\infty(E,n-1).
}
\end{equation}
Introduce the lifted distributions 
\begin{eqnarray}
\CC&\df& p^\ast(\C),\\
\DD&\df& n^\ast(\D),\\
\D&\df& \iota^\ast(\widetilde{\D}),
\end{eqnarray}
where $\D$ (resp., $\widetilde{\D}$) is the ($(n-1)$--dimensional) structural distribution on $J^\infty(E,n-1)$ (resp., $J^\infty(J^\infty(E,n),n-1)$).
%
 Notice that, unlike $\DD$, which has dimension $n-1$, both  $\CC$ and $\DD$ are infinite--dimensional, though they are well--behaved, in the sense that, homotopically, they are finite--dimensional.\par
Observe that $p$ maps a $\D$--leaf (i.e., a   Cauchy datum) $\Sigma_{(\infty)}$ into the small   leaf $\Sigma$ of $\C$. The same Cauchy datum is mapped by $n$ into the leaf $(\Sigma_0)_{(\infty)}$ of $\D$ (which is the corresponding Cauchy value). Hence, the secondary maps   
 \begin{eqnarray}
 (\III_{n-1}(\C), \D) &\stackrel{\p}{\longrightarrow} & (J^\infty(E,n),\C),\\
  (\III_{n-1}(\C), \D) &\stackrel{\n}{\longrightarrow} & (J^\infty(E,n-1),\D),
\end{eqnarray}
are well--defined. 
Recall that the inclusion $L\subseteq J^\infty(E,n)$ determines an inclusion $J^\infty(L,n-1)\subseteq J^\infty(J^\infty(E,n),n-1) $ (see Example \ref{exCarrierJet}).  
\begin{theorem}[Structural]\label{thStrutturale}
 Let $L$ (resp., $\Sigma'$) be a leaf of $J^\infty(E,n)$ (resp.,  $J^\infty(E,n-1)$). Then   the following identifications   \begin{eqnarray}
 p^{-1}(L)& = & J^\infty(L,n-1),\label{eqThStrutt1ok}\\
 n^{-1}(\Sigma') &=&J^\infty(\overrightarrow{n}|_{\Sigma_{(1)}}),\label{eqThStrutt2ok}
\end{eqnarray}
hold,  where $\Sigma$ is a Cauchy datum over $(\Sigma')_0$.

\end{theorem}

\begin{proof}
\REF{eqThStrutt2ok} is just a paraphrase of Proposition \ref{propAffineCheChissaSeFunziona}.\par
 Consider now the commutative diagram
 \begin{equation}
\xymatrix{
&\III_{n-1}(\C)\ar[d]\ar[dl]^p\ar@{^(->}[r]&J^\infty(J^\infty(E,n),n-1) \ar[d]^{\pi_{\infty,0}}\\
J^\infty(E,n)&\I_{n-1}(\C)\ar[l]^{p^\infty}\ar@{^(->}[r]& J^1(J^\infty(E,n),n-1) ,
}
\end{equation}
where $p^\infty$ is the limit of the $p^k$'s   and the vertical unlabeled arrow is ${\pi_{\infty,0}|_{\III_{n-1}(\C)}}$.

%

Observe that    planes in $J^1(L,n-1)$ are involutive, being contained in the tangent planes to $L$, which is $\C$--integral. Hence, $J^1(L,n-1)$ is also a subset of $\I_{n-1}(\C)$. Moreover,  
\begin{equation}
(p^\infty)^{-1}(L)=\{(T_\theta L, r)\mid r\subseteq T_\theta L,\quad \theta\in L\}
\end{equation}
is a bundle over $L$ whose fiber at $\theta\in L$ equals $\Gr(T_\theta L,n-1)$ (see Remark \ref{remNormBundleUno}), hence it coincides with $J^1(L,n-1)$. It remains to be observed that 
\begin{equation}\label{eqCheServeNelThEgr1}
\pi_{\infty,0}^{-1}(J^1(L,n-1))\cap \III_{n-1}(\C) =J^\infty(L,n-1),
\end{equation}
 where the latter is understood (see again Example \ref{exCarrierJet}) as a subset of $J^\infty(J^\infty(E,n),n-1) $. Inclusion \virg{$\supseteq$} is obvious, since $L$ is involutive and so are all its small submanifolds. Conversely, a point $[\Sigma']_\theta^\infty$ in the left--hand side of \REF{eqCheServeNelThEgr1} must be such that $ \theta=[L_0]^\infty_{\theta_0}$. On the other hand, being not maximal,  $\Sigma'$ must be contained into a leaf $L'_{(\infty)}$, where $L'$ has the same infinite jet as $L_0$ at $\theta_0$. So, there exist a small submanifold $\Sigma\subseteq L_0$,  which is tangent to infinite order to $\Sigma'$ at  $\theta_0$. Correspondingly,   a small submanifold (denoted by the same symbol) $\Sigma\subseteq L$ exists, such that $[\Sigma]_\theta^\infty=[\Sigma']_\theta^\infty$.

\end{proof}
Observe that, unlike \REF{eqThStrutt2ok}, \REF{eqThStrutt1ok} is canonical.
\begin{remark}[Coordinates]\label{remCoordInfCauchDat}
 By definition, the equations of $\III_{n-1}(\C)$   are the infinite prolongations of the equations of $\I_{n-1}(\C)$, which are just the \REF{eqEquazioniCoordinateDatiFinitiEQ}, rewritten with arbitrarily long multi--indexes, viz., 
 \begin{equation}\label{eqEquazioniCoordinateDatiFinitiEQpiulungo}
(u_{A,l}^\alpha)_a=u_{Aa,l}^\alpha+t_au_{A,l+1}^\alpha,\quad A\in\N_0^{n-1},l\in\N_0.
\end{equation}
It is a simple computation to show that all differential consequences of \REF{eqEquazioniCoordinateDatiFinitiEQpiulungo} read as
\begin{equation}\label{eqEquazioniCoordinateDatiFinitiEQpiulungoINF}
(u_{A,l}^\alpha)_B= \sum_{B'_1B'_2\cdots B'_s B''=B} t_{B'_1}t_{B'_2}\cdots t_{B'_s}u^\alpha_{AB'',l+s} ,\quad A,B\in\N_0^{n-1},l\in\N_0.
\end{equation}
In view of \REF{eqEquazioniCoordinateDatiFinitiEQpiulungoINF}, the  (infinite) set of functions
\begin{equation}\label{eqCoordInfCauData}
x^a,t,u^\alpha_{A,l}, t_B, \quad A,B\in\N_0^{n-1},B\neq\O,l\in\N_0, 
\end{equation}
can be taken as coordinates on $\III_{n-1}(\C)$, the infinite--order analog of  \REF{eqCoordPianPicInv}. By using \REF{eqCoordInfCauData} and  standard coordinates on $J^\infty(E,n)$, it looks obvious that $t_B$, with $B\in\N_0^{n-1}$, $B\neq\O$, are the fiber coordinates of $p$.\par
Now, similarly as for \REF{eqCoordinateGiusteSuDAtiFinitiBIS}, use  \REF{eqEquazioniCoordinateDatiFinitiEQpiulungoINF} to produce a new coordinate system
\begin{equation}\label{eqCoordInfCauDataALT}
x^a,t,u^\alpha_{A,0}, (u^\alpha_{\O,l})_B, t_B, \quad A,B\in\N_0^{n-1},A\neq\O,l\in\N, 
\end{equation}
from which one sees that $(u^\alpha_{\O,l})_B$, with $B\in\N_0^{n-1}$ and $l\in\N$ are the fiber coordinates of $n$. 
Coordinates \REF{eqCoordInfCauData} can be recovered from \REF{eqCoordInfCauDataALT} by the formulas
\begin{equation}\label{eqEquazioniCoordinateDatiFinitiEQpiulungoINF-alt}
u_{A,l}^\alpha= \sum_{B_1B_2\cdots B_s B=A}(-1)^s t_{B_1}t_{B_2}\cdots t_{B_s}(u^\alpha_{\O,l+s})_B ,\quad A\in\N_0^{n-1},l\in\N.
\end{equation}
\end{remark}

\begin{remark}[Affine case]
The infinite--order generalization of Lemma \ref{lemLocaleBellino} reads
 \begin{equation}\label{eqIsoLocalea}
\III_{n-1}(\C)\loc J^\infty(E,n)\times_{E_{-1}}J^\infty(E_{-1},n-1).
\end{equation}
Corresponcence \REF{eqIsoLocalea} takes a point $([\sigma]_x^\infty,[\Sigma]_x^\infty)$, where $\sigma$ is a section of $E\to E_{-1}$, to the point $j_\infty(\sigma)_\ast ([\Sigma]_x^\infty)$, where $j_\infty(\sigma)_\ast: J^\infty(E_{-1},n-1)\to J^\infty(J^\infty(E,n),n-1)$ is the jet map associated to $j_\infty(\sigma) $ (Definition \ref{defJetMap}). If coordinates \REF{eqCoordInfCauData} are split into $x^a,t,u^\alpha_{A,l}$ and 
$t_B$, one gets precisely the coordinates of a point in the right--hand side of \REF{eqIsoLocalea}.
\end{remark}

\begin{remark}
 A global analog of \REF{eqIsoLocalea} can be constructed by replacing, in the structural bundle $\C$ over $J^\infty(E,n)$, each fiber $\C_\theta$ by $J^\infty_0(\C_\theta,n-1)$. This shows that  $\III_{n-1}(\C) $ is an infinite--dimensional bundle over  $J^\infty(E,n)$, with generic fiber $J^\infty_0(\R^n,n-1)$, in strict analogy with flag bundles (Fact \ref{factFattoInteressante}).
\end{remark}

\begin{corollary}\label{corFondamentale}
 Let $L$ (resp., $\Sigma'$) be a leaf of $J^\infty(E,n)$ (resp.,  $J^\infty(E,n-1)$). Then   the following identifications of secondary manifolds 
 \begin{eqnarray}
 \p^{-1}(L)& = & J^\infty(L,n-1),\label{eqThStrutt1}\\
 \n^{-1}(\Sigma') &=&J^\infty(\overrightarrow{n}|_{\Sigma_{(1)}}),\label{eqThStrutt2}\\
(\III_{n-1}(\C), \CC) &=& J^\infty(E,n),\label{egThStrutt3}\\
 (\III_{n-1}(\C), \DD) &=& J^\infty(E,n-1),\label{egThStrutt4}
\end{eqnarray}
 hold,  where $\Sigma$ is a Cauchy datum over $(\Sigma')_0$.
\end{corollary}
\begin{corollary}[Transversality]\label{corTrasveralita}
 Projections $p$ and $n$ are  leafwise  transversal each other, i.e.,
\begin{enumerate}
\item $p$ projects diffeomorphically $n^{-1}(\Sigma)$ onto $J^\infty(E,n)$, for any leaf $\Sigma$ of $J^\infty(E,n-1)$;
\item $n$ projects diffeomorphically $p^{-1}(L)$ onto $J^\infty(E,n-1)$, for any leaf $L$ of $J^\infty(E,)$.
\end{enumerate}
\end{corollary}
\begin{proof}
 The second assertion is an immediate consequence of the fact that $J^\infty(L,n-1)$ is embedded into $J^\infty(E,n-1)$ (see the proof of Theorem \ref{thStrutturale}). For the first assertion, it is convenient to use the local coordinates from Remark \ref{remCoordInfCauchDat}. Namely, let $\Sigma$ be given by functions $f,g^\alpha$,
 \begin{equation}\label{eqTra1}
\Sigma: \left\{\begin{array}{c}t_B=\frac{\partial^B}{\partial x^B}f, \\u_A^\alpha=\frac{\partial^A}{\partial x^A}g^\alpha.\end{array}\right.
\end{equation}
Then $n^{-1}(\Sigma)$ is given, in the coordinates \REF{eqCoordInfCauDataALT}, by the same equations \REF{eqTra1}. Passing now to the coordinates \REF{eqCoordInfCauData},
 \begin{equation}\label{eqTra2}
n^{-1}(\Sigma): \left\{\begin{array}{ll}
u^\alpha_{A,l}=\frac{\partial^A}{\partial x^A}g^\alpha,& l = 0,\\
u^\alpha_{A,l} =   \sum_{B_1B_2\cdots B_s B=A}(-1)^s \frac{\partial^{B_1}}{\partial x^{B_1}}f \frac{\partial^{B_2}}{\partial x^{B_2}}f \cdots \frac{\partial^{B_s}}{\partial x^{B_s}}f(u^\alpha_{\O,l+s})_B, &l\neq 0,
 \end{array}\right.
\end{equation}
one sees that $n^{-1}(\Sigma)$ is parametrized by 
\begin{equation}\label{eqTra3}
(x^a,t,(u^\alpha_{\O,l+s})_B),
\end{equation}
 while the other coordinates are obtained via \REF{eqTra2}. So, the projection $p(n^{-1}(\Sigma))$ is  given by the same equations \REF{eqTra2}, in the standard coordinates $(x^a,t,u^\alpha_{A,l})$ of $J^\infty(E,n)$. Hence, $p(n^{-1}(\Sigma))$ is again parametrized by \REF{eqTra3}.
\end{proof}
\begin{remark}
\REF{eqThStrutt2} and \REF{eqThStrutt1}   might be seen as the secondary analog  of the $1\St$ order  projections of flag manifolds (see Fact \ref{factFattoInteressante} and Remark \ref{remNormBundleUno}).
\end{remark}
\section{Concluding remarks and perspectives}\label{secConclusiva}
\subsubsection*{{Secondary} ODEs}
Identifications \REF{egThStrutt3} and \REF{egThStrutt4} allow to regard $J^\infty(E,n)$ and\linebreak $J^\infty(E,n-1)$ as secondary quotients of the same secondary manifold $\III_{n-1}(\C)$. Indeed, $\CC$ can be understood as the distribution generated by $\D$ and by a $\p$--vertical secondary distribution (and similarly for $\DD$), as firstly pointed out by L. Vitagliano \cite{LucaPrivate}. Since the leaves of $\CC$ are canonically identified with the leaves of $\C$,   any equation in $n$ independent variables is the same  as a (secondary) distribution on the space of admissible Cauchy data. Such a perspective  seems to be   evidence of  a (formal) analogy with Hamiltonian formalism in mechanics.\par
\subsubsection*{{Twisted} characteristic cohomology}
Theorem \ref{thStrutturale} is the natural departing point to define a twisted generalization of the  characteristic cohomology of an equation (first of all, the empty one), where the coefficients    belong to    the $\p$-- or $\n$--vertical  characteristic cohomology of the corresponding space of Cauchy data, in analogy with the  differential Leray--Serre spectral sequence  associated with a fiber bundle. 
In particular, among terms of the  twisted characteristic cohomology it can be found the one which corresponds to an \virg{action--valued action}, i.e., an action integral whose value on a leaf $L$ (resp., $\Sigma$) is an action integral on $\p^{-1}(L)$   (resp., $\n^{-1}(\Sigma)$). In Section \ref{SecEsempioVariazionale} below we propose a toy model for such an action, and derive the corresponding Euler--Lagrange equations.\par
The theory of twisted characteristic cohomology should be a source of    simplification techniques in Calculus of Variations, and of  methods to   compute characteristic cohomology   of nonlinear PDEs, much as   the K\"unnet formula does in Algebraic  Topology.\par 
Fact \ref{FattoFattoso}, stemming from Theorem \ref{thStrutturale}, provides a basic understanding of the characteristic cohomology of the space of Cauchy data.
\begin{fact}\label{FattoFattoso}
 The $\D$--spectral sequence is 1--line.
\end{fact}
\begin{proof}
Embed  $J^\infty(E,n-1)$ into $J^\infty(E,n)$ (see Remark \ref{remarkCheQuiNonCentraNulla}), and then   observe that $\III_{n-1}(\C)$ is locally the space of horizontal infinite jets $\overline{J}^\infty(\overrightarrow{n}|_{J^\infty(E,n-1)_{(1)}})$ (see Proposition \ref{propAffineCheChissaSeFunziona}).
\end{proof}
\subsubsection*{Invariance of the framework}
From a mere set--theoretical point of view,    $ n^{-1}(\Sigma') $ is  but the inverse image of the submanifold $\Sigma_0'\subseteq E$ via the projection $\pi_{\infty,0}:J^\infty(E,n)\longrightarrow E$. The main virtue of   Theorem  \REF{thStrutturale} is to   reveal that $ n^{-1}(\Sigma') $ is   an empty equation, a fact which is essential if one is interested in   special subsets of $ n^{-1}(\Sigma') $,  which arise from the analysis of nonlinear PDEs, and compute their characteristic cohomology, by using the traditional geometrical and cohomological methods   for PDEs. In a sense, the whole machinery developed in this paper was aimed at the proof of \REF{eqThStrutt2ok}, but perhaps a key feature of our treatment was not given enough attention.  Namely,  the whole framework is \emph{invariant}, i.e., well--behaved with respect to transformations, which gives a total freedom in the choice of  coordinates for computational purposes (as in the toy model proposed in the last Section \ref{SecEsempioVariazionale}).\par
 \subsubsection*{Higher codimension and complete flags}

It is advisable to develop the theory for higher codimension flag jets, i.e., replace $n-1$ by any $n_0<n$ in the constructions presented here. The so--obtained formalism may have interesting applications , e.g., in the context of quasi--local Hamiltonians (see, e.g., \cite{Kio} concerning quasi--local mass in General Relativity). If  complete flags are taken as the departing point,   then the theory for the   twisted characteristic cohomology  of the so--obtained space of complete jet flag   should be particularly rich, and play the same role, in the context of nonlinear PDEs, as the CW--complexes  in Algebraic Topology.

\section{An applicative example}\label{SecEsempioVariazionale}
In view of Theorem  \REF{thStrutturale}, every leaf   of $J^\infty(E,n-1)$  produces an empty equation over the leaf itself,
\begin{equation}
\underset{\textrm{leaf of }J^\infty(E,n-1)}{\Sigma'}\longmapsto \underset{\textrm{ space of infinite jets of the infinite normal bundle}}{n^{-1}(\Sigma').}
\end{equation}
Moreover, thanks to Corollary \ref{corTrasveralita}, the empty equation $n^{-1}(\Sigma')$ can be seen as a closed subset of $J^\infty(E,n)$. Hence, if some equation and/or variational principle is imposed on $J^\infty(E,n)$, it will reflects on $n^{-1}(\Sigma')$. This phenomenon has been originally noticed by Vinogradov in 1984  (see \cite{Vin1984}, Section 8.5),  but its cohomological analysis was carried out in detail   by Vinogradov and the author in  the 2006 paper \cite{Mor2007} (see  also \cite{Mor2010}), where     the relationsip between the $\C$--spectral sequence associated with  ${n^{-1}(\Sigma')}$ and   the relative $\C$--spectral sequence of the surrounding jet space $J^\infty(E,n)$ is clarified.\par
%
The example developed below, which shows how a variational principle determines a natural equation on  ${n^{-1}(\Sigma')}$, is  also a case    where two action integrals of different horizontal degree are  summed up.\par
Suppose that $E$ is a closed domain in $\R^{n+m}$, such that an $m$--dimensional submanifold $G$ of  $\R^{n+m}$ exists, and $E$ is a tubular neighborhood of it. Then $E$ is  (globally) a bundle over $G$ with fiber $D^n$, and  (locally) a bundle over $D^n$ with fiber $\R^m$. Observe that the (graphs of the) sections of the latter belong to the larger  class of submanifolds
\begin{equation}
\A\df\{L\subseteq E\mid L\cap\partial E=\partial L, \  L\textrm{ is oriented and connected}\}.
\end{equation}
Put also
\begin{equation}
\partial\A\df \{\partial L\mid L\in\A\}.
\end{equation}
Observe that $\A$ is nothing but a subset of the space $J^\infty(E,n)$, made of leaves which are well--behaved with respect to integration (in the terminology of Calculus of Variations, they would referred to as \virg{admissible}, see also \cite{Vin1984}, Section 8.5 on this concern), and $\partial\A$ is a subset of the space $J^\infty(\partial E,n-1)$.\par
Let 
\begin{eqnarray}
\S&\in&\overline{H}^n(J^\infty(E,n),\pi_{\infty,0}^{-1}(\partial E)),\label{eqEsse}\\
\S_\partial &\in& \overline{H}^{n-1}(J^\infty(E,n-1) ),\label{eqEsseBordo}
\end{eqnarray}
 two action integrals, i.e., secondary real--valued functions on $\A$ and $\partial\A$, respectively,
 \begin{eqnarray}
\S:L\in\A&\longmapsto&j_\infty(L)^\ast\S\in H^n(L,\partial L)\cong \R,\\
\S_\partial:\Sigma\in\partial\A &\longmapsto& j_\infty(\Sigma)^\ast\in {H}^{n-1}(\Sigma )\cong\R,
\end{eqnarray}
 where the last identifications are an elementary fact of differential topology (see \cite{Bott}).
Toghether, \REF{eqEsse} and \REF{eqEsseBordo}, define a secondary function
 \begin{equation}
\A\ni L\ \stackrel{\S_{\textrm{tot}}}{\longmapsto}\   \S(L)+\S_\partial(\partial L)\in\R.
\end{equation}
Expectedly,  the set of critical points of $\S_{\textrm{tot}}$ is smaller than a mere (suitably defined) intersection  of the critical points of $\S$ and $\S_\partial$, because an \virg{interaction term} arises. Namely, for any $L\in\A$, consider the module of cosymmetries (see \cite{Sym})
\begin{equation}
(\varkappa^\dag)_L\df\varkappa^\dag(p(n^{-1}((\partial L)_{(\infty)})))
\end{equation}
of $p(n^{-1}((\partial L)_{(\infty)}))$,    the canonical splitting 
\begin{equation}
\varkappa^\dag(J^\infty(E,n),\pi_{\infty,0}^{-1}(\partial E))\loc\varkappa^\dag(J^\infty(E,n))\oplus (\varkappa^\dag)_L,
\end{equation}
 and the corresponding decomposition\footnote{%
In   \cite{Vin1984}, Section 8.5, equations $(\ddd_{\textrm{rel}}\S)_L$ are denoted by  $\Gamma(\overline{\omega})$, where $\overline{\omega}$ is a representative of $\S$, while in  \cite{Mor2007}   they are denoted by $\theta'_{\overline{\omega}}$.} of the relative Euler--Lagrange differential\footnote{Introduced in  \cite{Mor2007},  Sec. 3.4, where it is denoted by $\boldsymbol{E}_{\textrm{\normalfont rel}}$.} of $\S$
\begin{equation}
\ddd_{\textrm{rel}}\S=(\ddd\S, (\ddd_{\textrm{rel}}\S)_L).\label{eqCanSplittRelDiffSec}
\end{equation}
It turns out that $(\ddd_{\textrm{rel}}\S)_L=0$ is a differential equation in $p(n^{-1}((\partial L)_{(\infty)}))$, i.e., imposed on the sections of $\overrightarrow{n}|_{((\partial L)_{(\infty)})_{(1)}}$, which formalizes precisely the above idea of interaction.

\begin{theorem}
\begin{equation}\label{eqGasanteLagrangianeCheInteragiscono}
 L\textrm{ is critical for }\S_{\textrm{\normalfont tot}}\Leftrightarrow\left\{\begin{array}{rll}L_{\infty}&\in\{\ddd\S=0\}^{(\infty)}&\subseteq J^\infty(E,n), \\ 
 L_{(\infty)}\cap\pi_{\infty,0}^{-1}(\partial L)&\in \{(\ddd_{\textrm{\normalfont rel}}\S)_L=0\}^{(\infty)}&\subseteq J^\infty(\overrightarrow{n}|_{((\partial L)_{(\infty)})_{(1)}}) ,\\ 
 (\partial L)_{(\infty)}&\in \{\ddd\S_\partial=0\}^{(\infty)}&\subseteq J^\infty(\partial E,n-1).\end{array}\right.
\end{equation}

\end{theorem}
\begin{proof}
 Obviously, $L$ is critical for $\S_{\textrm{\normalfont tot}}$ if an only if $L$ is a solution of the relative Euler--Lagrange equations,
 \begin{equation}\label{eqDiffEulRel}
\ddd_{\textrm{rel}}\S=0,
\end{equation}
and $\partial L$ is a solution of $\ddd\S_\partial=0$, i.e., the last equation of the list \REF{eqGasanteLagrangianeCheInteragiscono}. It remains to observe that the first two equations are synthetically expressed by \REF{eqDiffEulRel}, thanks to \REF{eqCanSplittRelDiffSec}.
\end{proof}
When \REF{eqEsse} and  \REF{eqEsseBordo} are volume integrals,   critical points of $\S_{\textrm{tot}}$ are  the  \emph{least--volume and least--boundary---area} submanifolds of $E$.\par
Apparently,   \REF{eqGasanteLagrangianeCheInteragiscono} is just a clean way to write down the so--called natural boundary   conditions in the Calculus of Variations, and all the machinery exploited to obtain  \REF{eqGasanteLagrangianeCheInteragiscono} is  but a paraphrase of the classical analytical manipulations on variational integral (see, e.g.,    \cite{Giaq,Brunt}) exploited to derive the natural boundary conditions. In fact,  a very important feature of \REF{eqGasanteLagrangianeCheInteragiscono}, their invariance, does not show at a superficial look. Such a property allows, for instance, to derive the correct expression of the transversality conditions, for any \virg{tubular} manifold---which are not known to date---just  by a wise choice of coordinates.  We will not go into the details of the general construction, but present a simple toy model with $n=m=1$.
\begin{example}[The problem of Columbus]\label{exColProb}
Given the curves $\Gamma_1$ and $\Gamma_2$ in $\R^2$, consider the problem of finding, among the (non self--intersecting) (smooth)  curves which start from a point of $\Gamma_1$ and end to a point of $\Gamma_2$ (without crossing $\Gamma_1\cup\Gamma_2$ in any other point), those whose length is (locally) minimal. Obviously, a    curve $\gamma$  is a solution of  the problem at hand if and only if
\begin{itemize}
\item[(EL)] $\gamma$ is a   {straight line};
\item[(TC)] $\gamma$   {hits at a right angle $\Gamma_1\cup\Gamma_2$}.
\end{itemize}
\end{example}
The problem can be formalized by means of a  Lagrangian density  $fdx$, where $f=f(x,y,y')$, on a tubular submanifold $E\subseteq\R^2$, with   $\partial E=\Gamma_1\cup\Gamma_2$. In this setting, equations (TC) for a curve $\gamma=(x,y(x))$ read
 \begin{equation}\label{eqTCsemplici}
\left(f-\frac{\partial f}{\partial y'}\right) x^\Gamma + \frac{\partial f}{\partial y'} y^\Gamma =0,
\end{equation}
where $(x^\Gamma,y^\Gamma)$ is a vector tangent to $\partial E$.
\begin{proof}
Equation \REF{eqTCsemplici} can be obtained in few lines (see \cite{Brunt}). We propose an alternative way, which stresses the role of invariance of \REF{eqGasanteLagrangianeCheInteragiscono}. 
To this end, choose  a diffeomorphism between $E$ and the cylinder $[0,1]\times\R$, and denote by $\omega=gdx$ the pull-back of $fdx$ to such a cylinder. Then $\S:=[\omega]$ is an element of $\overline{H}^1(J^\infty(\pi),\pi_{\infty}^{-1}(\{0,1\}))$, where $\pi:[0,1]\times\R\to [0,1]$,
and 
 \begin{equation}\label{eqTCsuCilindro}
(\ddd_{\textrm{rel}}\S)_{\{0,1\}}=\left.\frac{\partial g}{\partial y'}\right|_{\pi_{\infty}^{-1}(\{0,1\})}.
\end{equation}
By pulling back \REF{eqTCsuCilindro} on $E$, one obtains \REF{eqTCsemplici}.
\end{proof}

\subsection*{Acknowledgments}
We belong to a  privileged community   which can play with theoretical nonsense  without caring about life,   thanks to the hard  work of many good people whose   existence is often forgotten---to them goes the author's deepest   gratitude. He is thankful for the indispensable hints and nudges coming from   L. Vitagliano, who first proposed Definition \ref{defSpazioDatiCauchy} for $k=\infty$, and M. B\"achtold, with whom he started the study of the   variational problems with free boundary which describe the  \virg{flight of dead vipers}, and for the friendly and stimulating environment of the Silesian University in Opava. It is his pleasure to thank to the the Grant Agency of the Czech Republic (GA \v CR)
for financial support under the project P201/12/G028.

 \def\cprime{$'$} \def\cprime{$'$} \def\cprime{$'$}

\end{document}